\documentclass[letterpaper,twoside,11pt,leqno]{article}
\usepackage{amsmath,amsthm,amssymb, enumerate,url}
            \newcommand{\marginalnote}[1]{}
\usepackage[all]{xy}
\usepackage{graphicx}
\usepackage{epsfig}
\usepackage{color}	

\swapnumbers \theoremstyle{plain}

\newtheorem{Thm}{Theorem}[section]
\newtheorem{Prop}[Thm]{Proposition}
\newtheorem{Lem}[Thm]{Lemma}
\newtheorem{Cor}[Thm]{Corollary}

\newtheorem*{Thm*}{Theorem}
\theoremstyle{remark}

\theoremstyle{definition}
\newtheorem{Rem}[Thm]{Remark}
\newtheorem{Ex}[Thm]{Example}

\newtheorem{Def}[Thm]{Definition}

\newtheorem{Rev}[Thm]{Review}
\numberwithin{figure}{Thm}

\newcommand{\lconj}[2]{{}^{#1}\mkern-1mu{#2}}
\newcommand{\rconj}[2]{{}^{#1}\mkern-1mu{#2}}
\newcommand{\gen}[1]{\left\langle#1\right\rangle}

\newcommand{\gp}[2]{\gen{#1\mid #2}}
\newcommand{\prs}[2]{\gen{#1\parallel #2}}
\newcommand{\ol}{\overline}
\newcommand{\Z}{\mathbb{Z}}

\newcommand{\C}{\mathbb{C}}

\def\coloneqq{\mathrel{\mathop\mathchar"303A}\mkern-1.2mu=}

\DeclareMathOperator{\Tor}{Tor}

\DeclareMathOperator{\uE}{\underline{E}}
\DeclareMathOperator{\mor}{mor}
\DeclareMathOperator{\Hop}{H}
\DeclareMathOperator{\Mod}{Mod}
\begin{document}

\title{On the classifying space for proper actions of groups with cyclic torsion}
\author{Yago Antol\'{i}n and Ram\'{o}n Flores}

\maketitle

\begin{abstract}

In this paper we introduce a common framework for describing the topological part of the Baum-Connes conjecture for a wide class of groups. We compute the Bredon homology  for groups with aspherical presentation, one-relator quotients of products of locally indicable groups, extensions of $\mathbb{Z}^n$ by cyclic groups, and fuchsian groups. We take advantage of the torsion structure of these groups to use appropriate models of the universal space for proper actions which allow us, in turn, to extend some technology defined by Mislin in the case of one-relator groups.

\medskip

{\footnotesize

}

\end{abstract}

\section{Introduction}

In \cite{MislinValette},  Mislin computed the Bredon homology of one-relator groups with coefficients in the complex representation ring. These homology groups were defined in the sixties by  Bredon in the context of equivariant Homotopy Theory. Since the statement by Baum-Connes of their famous conjecture (see \cite{BaumConnes} for a thorough account and section 5 here for a quick review), there has been a growing interest in the computation of the Bredon homology groups, as they give, via a spectral sequence, a very close approximation to the topological part of the conjecture. Moreover, they are reasonably accessible from the point of view of the computations.

Let $G$ be a discrete group. To deal with the topological part of Baum-Connes conjecture, it is necessary to recall some basics of the theory of proper $G$-actions.  A \emph{model} \emph{for the classifying space for proper G-actions} $\uE G$ is a $G$-CW-complex $X$ with the property that,  for each subgroup $H$ of $G,$ the subcomplex of fixed points is contractible if $H$ is finite, and empty if $H$ is infinite. The latter condition means precisely that all cell stabilizers are finite, and, in this case, we say that the $G$-action is \emph{proper}. The model $X$ for $\uE G$ is important in our context because it is the target of the topological side of Baum-Connes conjecture. We also denote by $(X)^{sing}$ the \emph{singular part} of $X,$ that is, the subcomplex consisting of all points in $X$ fixed by some non-trivial element of $G,$ and we say that $(X)^{sing}$ is a model for  $(\uE G)^{sing}.$ It is worth to note that both models for $\uE G$ and $(\uE G)^{sing}$ are well-defined up to $G$-homotopy equivalence.

In this paper we extend Mislin's result \cite[Corollary 3.23]{MislinValette} to a wider class of groups. The key observation here is that the computation of the Bredon homology of one-relator groups does not use in full potential the shape of the concrete relation, but it relies only in two facts: the existence (up to conjugation) of a unique maximal finite subgroup, and the construction of a model for $\uE G$ whose singular part is zero-dimensional.

We consider here the class $\mathcal{G}_{cct}$ of groups for which there is a finite family of cyclic subgroups such that every non-trivial torsion element of the group belongs to exactly one member of the family up to unique conjugation. If we consider the classical model $X$ for $\uE G$ as a bar construction for a certain $G\in \mathcal{G}_{cct}$, it is easy to see (see Proposition \ref{bar-construction}) that for every non-trivial finite subgroup $H<G$ the fixed-point set $(X)^H$ is just a vertex of the model. Hence the singular part of $X$ is zero-dimensional, and we have all the needed assumptions to compute the Bredon homology groups (Proposition \ref{T:clabredonhom}) and hence the Kasparov $KK$-groups (Proposition \ref{Kasparov}).

Aside from the computations, the other main achievement of our article is the identification of well-known families of groups which belong to the class $\mathcal{G}_{cct}$: groups with an aspherical presentation, one-relator products of locally indicable groups, some extensions of $\mathbb{Z}^n$ for cyclic groups, and some fuchsian groups. For the first two families, moreover, we describe some particular models of $\uE G$ which turn sharper our homology computations. For the particular class of Hempel groups (see Definition \ref{D:hempel}), we show that they have Cohen-Lyndon aspherical presentations, so we use Magnus induction and hierarchical decompositions to prove that Baum-Connes hold; then, our methods are also valid to compute analytical $K$-groups $K_i^{top}(C^*_r(G))$ in this case. Note that the validity of the conjecture is not known for more general classes of groups of cohomological dimension two.

The paper is structured as follows: in section 2 we formally introduce the class $\mathcal{G}_{cct}$, as well as some background that will be needed in the rest of the paper; in section 3 we present some families of groups in the class, including some particular models of the classifying space for proper actions and interesting relationships with surface groups and other one-relator groups; section 4 is devoted to Bredon homology, which we describe in a little survey before undertaking our computations, and we finish in section 5 with the computation of the topological part of the Baum-Connes conjecture for these groups.

\textbf{Acknowledgments}.
Part of this paper is based on the Ph.D. thesis of the first author at the Universitat Auton\`{o}ma de Barcelona. The first author is grateful to Warren Dicks for his help during that period.

We are grateful to Ruben Sanchez-Garc\'{i}a for many useful comments and observations, and also to  Brita E. A. Nucinkis, Ian Leary, Giovanni Gandini and David Singerman for helpful conversations.

The authors were supported by MCI (Spain) through project MTM2008-01550 and EPSRC through project EP/H032428/1 (first author) and project MTM2010-20692 (second author).

\section{The class of groups $\mathcal{G}_{cct}$}

In this article we will deal with groups that have, up to conjugation, a finite family of maximal malnormal cyclic subgroups. Precisely we deal with groups satisfying the following condition:

\begin{enumerate}
\item[(C)] There is a finite family of non-trivial  finite cyclic subgroups $\{G_\lambda\}_{\lambda\in \Lambda}$ such that for each non-trivial torsion subgroup $H$ of $G,$ there exists a unique $\lambda\in \Lambda$ and a unique coset $gG_\lambda \in G/G_{\lambda}$ such that $H\leqslant g G_\lambda g^{-1}.$\end{enumerate}

In particular, the groups $G_{\lambda}$ are maximal \emph{malnormal} in $G$. Recall that a subgroup $H$ of a group $G$ is {\it malnormal} if $H\cap gHg^{-1}=\{1\}$ for all $g\in G-H$.
The class of  groups $G$ satisfying the condition (C) will be denoted by $\mathcal{G}_{cct}$.

Observe that any finite cyclic group and any torsion-free group is in $\mathcal{G}_{cct}.$
Moreover, this  class is also closed by free products, so for example, the infinite dihedral group is in  $\mathcal{G}_{cct}.$
In the following section we will describe many interesting examples.

Our main objective is to describe  the Bredon homology for a group $G$ in $\mathcal{G}_{cct}$. Our approach to this computation will be through the classifying space for proper $G$-actions, so we recall here a classical model which turns out to be very useful for our purposes. More sophisticated and particular models for $\uE G$ will appear later in the article.

\begin{Prop}\label{bar-construction}
Let $G$ and $\{G_\lambda\}_{\lambda\in \Lambda}$ satisfying (C), then there exists a model $\mathcal{C}$ for $\uE G$ with $\dim (\mathcal{C}^{sing})=0.$
\end{Prop}
\begin{proof}
We adapt the proof of \cite[Proposition 8]{PinedaLeary} to our context.
Let $X$ be the left $G$-set $\{g G_\lambda :\lambda \in \Lambda, g\in G\}.$  A subgroup $H$ of $G$ fixes $g G_\lambda\in X$ if and only if $g^{-1}Hg\subseteq G_\lambda$, so $H$ is finite and, by condition (C), $H$ fixes exactly one element of $X.$

Let $E$  a model for the universal space $EG$ (for example \cite[Example 1B.7]{Hatcher}). Recall that the join $X*E$ is the quotient space of $X\times E\times [0,1]$ under the identifications $(x,e,0)\sim(x,e',0)$ and $(x,e,1)\sim(x',e,1).$ The product $X\times E\times [0,1]$ is a $G$-set with $G$ acting trivially in the interval $[0,1]$ and induces an action on $X*E.$

Let $H$ be a subgroup of $G.$ If $H$ is infinite, then $H$ fixes no point of $X$ or $E.$ If $H$ is non-trivial finite, it acts freely on $X\times E\times (0,1],$ and fixes exactly one point of $X\times E\times \{0\}.$ Hence $(X*E)^{H}$ is contractible and 0-dimensional. If $H$ is trivial, then it fixes $X*E$, which is contractible since $E$ is contractible. Thus, $\mathcal{C}=X*E$ is a model for $\uE G,$ such that $\dim (\mathcal{C}^{sing})=0.$
\end{proof}

\begin{Rem}

Our class of groups $\mathcal{G}_{cct}$ is a subclass of the groups with appropriate maximal finite subgroups considered in \cite[4.11]{Luck04}. A particular model for $\uE G$ is also provided there.

\end{Rem}

A useful tool to show that a group $G$ and a family of subgroups $\{G_\lambda\}_{\lambda \in \Lambda}$ of $G$ satisfy the condition (C) is the following theorem:
\begin{Thm}[{\cite[Theorem 6]{HowieSchneebeli}}]\label{Thm:Howie} Let $G$ be a group and $\{G_\lambda\}_{\lambda \in \Lambda}$ a finite family of finite subgroups. If there exists an exact sequence of $\Z G$-modules
$$0\to \bigoplus_{\lambda\in \Lambda} \Z[G/G_\lambda]\oplus P\to P_{n-1}\to \dots\to P_0\to \Z $$
where $P, P_{n-1},\dots, P_0$ are $\Z G$-projective, then for every finite subgroup $H$ of $G$, there exists a unique $\lambda \in \Lambda$ and a unique $gG_\lambda\in G/G_\lambda$ such that $H$ is contained in $gG_\lambda g^{-1}.$
\end{Thm}

We remark that Alonso \cite{Alonso} find bounds for the dimension of the $\uE G$ of groups satisfying the hypothesis of the previous theorem.
\section{Examples of groups in the class $\mathcal{G}_{cct}$}

In this section we introduce some families of groups which will be proved to be in the class $\mathcal{G}_{cct}$. Let us fix some notation first.

\textbf{Notation}. We find useful to have different notations for a group given by a presentation and the presentation itself.  We use a double bar $\prs{X}{R}$ to distinguish a presentation from the
group being presented $\gp{X}{R}.$

Let $G$ be a group and $r\in G$. If $g$ lies in a unique maximal infinite cyclic subgroup $C$ of $G,$ we will denote by $\sqrt[G]{g}$ the unique generator of $C$ for which $g$ is a positive power. In this event, if $g$ is the $n$-th power of $\sqrt[G]{g},$ we denote $n$ by $\log_G (g).$

If $X$ is a subset of $G,$ we denote by $\lconj{G}{X}$ the image of $X$ under left-conjugation by $G,$ that is, $\lconj{G}{X}=\{gxg^{-1}: x\in X, g\in G\}.$ When $X=\{x\}$ we usually write $\lconj{G}{x}$ instead of $\lconj{G}{\{x\}}.$

\subsection{Groups with aspherical presentation}\label{ss:aspherical}

Let $F$ be a free group freely generated by a finite set $X=\{x_1,\dots, x_n\},$ let $R$ be a subset of $F$ and $G\coloneqq \gp{X}{R}.$ In a free group, each element lies in a unique maximal infinite cyclic subgroup of $F$. For $r\in R,$ let $G_r$ be the image of $\gen{\sqrt[F]{r}}$ in  $G$, a finite cyclic subgroup of $G.$

Recall that a CW-complex is {\it aspherical} if its universal covering is contractible. By the Hurewicz-Whitehead theorem a CW-complex is contractible if and only if it is acyclic and simply connected.

There exist several concepts of aspherical presentations, see \cite{CCH}. We will say that a presentation $\prs{X}{R}$ is {\it aspherical} if the abelianized of $\gen{\lconj{F}{R}}$ is isomorphic to $\oplus_{r\in R} \Z[G/G_r]$, and then  a group is aspherical if it admits an aspherical presentation. It is a famous conjecture of Eilenberg-Ganea that the torsion-free aspherical groups are precisely the groups of cohomological dimension two. We now review the topological significance of asphericity.

Recall that the Cayley graph of $G$ with respect to $X$ is a $G$-graph $\Gamma$ with vertex set $G$ and edge set $G\times X$; for an edge $e=(g,x)$ the initial vertex $\iota e$ is $g$ and the terminal vertex $\tau e$ is $g\cdot x.$
The augmented cellular chain complex of $\Gamma$ is
$$\Z [G\times X]\stackrel{\partial}{\to} \Z G\to \Z \to 0$$
where $\partial(g,x)=gx-g.$

It is well-known that the kernel of $\partial$ is isomorphic to $\gen{\lconj{F}{R}}_{\text{ab}}$
and the kernel map $\theta\colon\gen{\lconj{F}{R}}_{\text{ab}}\to  \Z[G\times X]$ is induced by the total free derivative $\frac{\partial}{\partial X}\colon F\to \Z[F\times X],$ which is defined by $f\mapsto (\frac{\partial f}{\partial x_1},\dots, \frac{\partial f}{\partial x_n}).$ The map $f\mapsto \frac{\partial{f}}{\partial x_i}$ is a derivation from $F\to \Z F,$ i.e. it satisfies the identity $\frac{\partial f_1f_2}{\partial x_i}=\frac{\partial f_1}{\partial x_i}+f_1\frac{\partial f_2}{\partial x_i}.$ Hence $\frac{\partial}{\partial x_i}$ is uniquely determined by its values on $X,$ and $\frac{\partial x_j}{\partial x_i}$ is equal to $0$ if $i\neq j$ and $1$ if $i=j.$ See \cite[Proposition 5.4]{Brown} or \cite[Corollary 9.4]{DicksDunwoody} for a proof.

Hence, there is the following exact sequence of $\Z G$-modules

\begin{equation}\label{eq:sec}
0\to \gen{{\lconj{F}{R}}}_{\text{ab}}\to \Z [G\times X]\to \Z G \to \Z \to 0.
\end{equation}

From now on we assume that $\prs{X}{R}$ is aspherical, that is $\gen{\lconj{F}{R}}_{ab}\cong\oplus_{r\in R} \Z [G/G_r].$ This is the case, for example, when $R$ consists of a single element by Lyndon's identity theorem (see, \cite{Lyndon}). Let $R_0\subset R$ be the set of  $r\in R$ for which $G_r\neq 1$. Then, by Theorem \ref{Thm:Howie}, $G$ and $\{G_r\}_{r\in R_0}$ satisfy the  condition (C).

With our assumption the sequence \eqref{eq:sec} becomes
\begin{equation}\label{eq:seq}
0\to \oplus_{r\in R} \Z[G/G_r]\stackrel{\theta}{\to} \oplus_{x\in X} \Z G\stackrel{\partial}{\to }\Z G\to \Z \to 1.
\end{equation}

We describe now a model for $\uE G$, which is built exactly the same way as the usual for one-relator groups. The construction is basically the same as in \cite[Review 7.4]{ADL} which deal with a special case when $|R|=2$, so we omit the details.

Recall that the {\it Cayley complex } of $\prs{X}{R},$ denoted $\overline{\mathcal C}= \overline{\mathcal C}\prs{X}{R},$ is a two-dimensional CW-complex with exactly one 0-cell denoted $[1]$, with set of 1-cells, denoted $[X],$ in bijective correspondence with $X$ by a map denoted $X\to [X],$ $x\mapsto [x],$ and with set of 2-cells, denoted $[R]$, in bijective correspondence with $R$ by a map denoted $R\to [R],$ $r\mapsto [r].$ The attaching maps are determined by the $1$-cells. Each $r$ is a word in $X^{\pm 1},$ and we take the closure of the $2$-cell $[r]$ to be a polygon whose (counter-wise) boundary has the corresponding labeling in the $1$-cells and their inverses, and this labeling gives the attaching map for $[r];$ the inverse of a 1-cell is the same 1-cell with the opposite orientation. The fundamental group of $\overline{\mathcal C},$ with base-point the unique 0-cell, has a natural identification with $G=\gp{X}{R}.$

Let $\mathcal{C}$ be the universal cover of $\overline{\mathcal{C}}.$ The $1$-skeleton of $\mathcal{C}$ is the Cayley graph of $G$ with respect to $X.$ It can be checked that $\mathcal{C}$ is simply connected and the augmented cellular complex of $\mathcal{C}$ is the $\Z G$-complex
$$0\to \Z [G\times [R]]{\to} \Z [G\times [X]]\to \Z[G\times\{[1]\}]\to \Z \to 0.$$

For $[r]\in [R],$ let $(g,[r])$ be the lift of the 2-cell of $[r]$ in $\mathcal{C}$ corresponding to the vertex $g\in G.$ Let $\mathcal{C'}$ denote the CW-complex that is obtained from $\mathcal C$ by identifying the $2$-cells $(g,[r])$ and $(g\sqrt[F]{r},[r]),$ for each $g\in G$ and $r\in R.$ We denote this identified cell by $(gG_r,[r]).$ It can be checked that $\mathcal{C'}$ is again simply connected, and the augmented chain complex of $\mathcal{C'}$ is the exact sequence \eqref{eq:seq}. Hence $\mathcal{C'}$ is contractible.

If $G_r=\{1\}$ for each $r\in R,$ that is $R=\sqrt[F]{R},$ then $\mathcal{C}=\mathcal{C'}$ is acyclic and hence contractible; in particular $\overline{\mathcal{C}}$ is aspherical and $\mathcal{C}$ is an $\uE G$.

If $\sqrt[F]{R}\neq R,$ $G$ does permute the open cells of $\mathcal C'$, but $\mathcal C'$ is not a $G$-CW-complex since, for some $r\in R,$ $\sqrt[F]{r}$ fixes the 2-cell that is in the equivalence class of $(1,[r])$ but does not fix it pointwise since it does not fix the 1-cells where this cell is attached.

Let $\mathcal{C}''$ denote the CW-complex obtained from $\mathcal C'$ by subdividing each 2-cell $(gG_r,[r])$ into $\log_F r$ 2-cells. If we think $(gG_r,[r])$ as a  polygon with $|r|$-sides, the subdivision corresponds to adding a new vertex in the center, dividing from this vertex into $\log_F r$ subpolygons, such that $\sqrt[F]{r}$ permutes them.

Notice that $\mathcal{C}''$ has one free $G$-orbit of $2$-cells for each $r\in R.$ Also, for each $r\in R$, we have also added a free $G$-orbit of 1-cells to $\mathcal{C}'.$ Finally for each $r\in R$ we have added a $G/G_r$-orbit of $0$-cells to $\mathcal{C'}.$  The CW-complex $\mathcal{C}''$ is then a contractible $G$-CW-complex whose augmented cellular chain complex  is exact and has the form
\begin{equation}\label{eq:seqaspherical}
0\to \Z G^{|R|}\to \Z G^{|R|+|X|}\to \Z G \oplus (\oplus_{r\in R} \Z[G/G_r])\to \Z \to 0.
\end{equation}

Now by (C) for each non-trivial torsion subgroup $H$ of $G$ there exists a unique $r\in R$ and $gG_r\in G/G_r$ such that  $H\leq \lconj{g}{G_r},$ and therefore $H$ fixes only the $0$-cell obtained on the subdivision of the $2$-cell $(gG_r, [r])$ of $\mathcal{C'}.$ Then $\mathcal{C}''$ is an $\uE G.$

\subsection{One-relator products of locally indicable groups}\label{ss:oreli}

Recall that a group is {\it indicable} if either is trivial or it has an infinite cyclic quotient. A group is {\it locally indicable} if every finitely generated subgroup is indicable. Notice that a locally indicable group is torsion-free.

Let $A,B$ be locally indicable groups having finite dimensional Eilenberg-MacLane spaces $\ol{\mathcal{C}}_A$ and $\ol{\mathcal{C}}_B$ respectively. Let $r\in A*B,$ such that $r$ is not conjugate to an element of $A$ nor of $B.$ It can be deduced, from Bass-Serre theory, that the centralizer $\mathbf{C}_{A*B}(r)$ is infinite cyclic, and hence we can define $\sqrt[A*B]{r}.$ Denote $G=(A*B)/\gen{ \lconj{A*B}{r} }$ and $G_r=\gen{\sqrt[A*B]{r}}/\gen{r}\leq G.$

Let now $\ol{\mathcal{C}}_{A*B}$ be the $CW$-complex obtained by attaching a $1$-cell $\mathbf{e}$ to the disjoint union of $\ol{\mathcal{C}}_A$ and $\ol{\mathcal{C}}_B,$ where the endpoints are $0$-cells in $\ol{\mathcal{C}}_A$ and $\ol{\mathcal{C}}_B$; then $\ol{\mathcal{C}}_{A*B}$  is connected and has the homotopy type of an Eilenberg-MacLane space $K(A*B,1).$ Choose a map $\phi: S^1\to \ol{\mathcal{C}}_{A*B}^{(1)}$ that represents $\sqrt[A*B]{r},$ where $\ol{\mathcal{C}}_{A*B}^{(1)}$ denotes the $1$-skeleton of $\ol{\mathcal{C}}_{A*B}.$ We will assume that the basepoint for $S^1,$ goes under $\phi$ to a vertex $\mathbf{v}$ of $\ol{\mathcal{C}}_{A*B}$.

Let $\ol{\mathcal{C}}_{G_r}$ be a model for $K(G_r,1).$ If $G_r=1$ we may think of $\ol{\mathcal{C}}_{G_r}$ as a disk and $p\colon S^1\to \ol{\mathcal{C}}_{G_r}$ as the natural inclusion to the boundary. If $G_r$ is a non-trivial cyclic group of order $\log_{A*B}(r),$  $\ol{\mathcal{C}}_{G_r}$ is a CW-complex with one cell in each dimension and $p\colon S^1\to \ol{\mathcal{C}}_{G_r}$ is the natural projection.

In \cite[Theorem 1]{Howie84}, it is showed that the following push-out (Figure \ref{fig:pushout}) has the homotopy type of $K(G,1)$.
\begin{figure}[ht]
\centerline{
\xymatrix{
S^1\ar[r]^{\phi}\ar[d]_p& \ol{\mathcal{C}}_{A*B} \\
\ol{\mathcal{C}}_{G_r}
}
}\caption{The push-out that gives a $K(G,1)$.}\label{fig:pushout}
\end{figure}
In the sequel, the push-out of this diagram will be denoted $\ol{\mathcal{C}}_G.$

Using this construction, in \cite[Proposition 7]{Howie84} Howie shows that there is a sequence of $\Z G$-modules
$$0\to \Z[G/G_r]\oplus P\to P_{n-1}\to \dots\to P_0\to \Z $$
where $P, P_{n-1},\dots, P_0$ are $\Z G$-projective, and thus the subgroup $G_r$ and $G$ satisfy the condition (C) by Theorem \ref{Thm:Howie}.

 We now describe how to obtain an $\uE G$ from $\ol{\mathcal{C}}_G$   with $\dim \uE G^{sing}\leq 0.$ Let $\mathcal{C}_G$ denote the universal cover of  $\ol{\mathcal{C}}_G.$ Observe that if $\sqrt[A*B]{r}=r,$ then $G$ is torsion-free, and we take $\mathcal{C}_G$ as our model for $\uE G$.

Assume then $\sqrt[A*B]{r}\neq r,$ so we will modify $\mathcal{C}_G$ to have all the singular action in dimension 0.
Let $\ol{\mathbf{\alpha}}$ be the 2-cell of  $\mathcal{C}_{G_r}$ and let $\mathbf{\alpha}$ denote a lift of $\ol{\mathbf{\alpha}}$ in $\mathcal{C}_G.$
For $n> 2$ we remove $G$-equivariantly the orbit of the $n$-dimensional cells of $\mathcal{C}_G$ corresponding to the $n$-dimensional cell of $\ol{\mathcal{C}}_{G_r},$ and we identify two $2$-cells $g\mathbf{\alpha}$ and $g'\mathbf{\alpha}$ if $g'\in gG_r.$ Notice that this two cells have the same boundary. We denote the $2$-cell identified with $\mathbf{\alpha}$ by  $G_r{\mathbf{\alpha}} $, and we claim that the obtained space, denoted by $\mathcal{C}_G',$ is still contractible.

The latter construction is not a $G$-$CW$-complex, because $\sqrt[A*B]{r}$ fixes the $2$-cell $G_r{\mathbf{\alpha}} $  but does not fix it pointwise since it does not fix the 1-cells where it is attached. So, to obtain a $G$-$CW$-complex, we subdivide each 2-cell $gG_r\mathbf{\alpha}$ like in the case of aspherical presentations. Formally, we remove the $G$-orbit of  $G_r\mathbf{\alpha}$ and  we add a $G/G_r$-orbit of $0$-cells $\{g G_r \mathbf{u} : g\in G\}$, and  a $G$-orbit of $1$-cells from $\{g\mathbf{f}: g\in G\}$  with $\mathbf{f}$ attached to $g \mathbf{v}$ to $gG_r \mathbf{u}$ and finally a $G$ orbit of $2$-cells $\{g\mathbf{\beta}:g\in G\}$ with $g\beta$ attached to the path that starts at $g\sqrt[A*B]{r} \mathbf{v},$ then goes through the edge $g\sqrt[A*B]{r} \mathbf{f,}$ then goes through $g \mathbf{f}^{-1}$ and finally through the subpath of the lift of $\phi$ that goes from $g\mathbf{v}$ to $g\sqrt[A*B]{r}\mathbf{v}.$

We denote this new complex by $\mathcal{C}_G''.$ Clearly it is a $G$-$CW$-complex. Let $H$ be a finite subgroup of $G,$ by (C) there exists a unique $gG_r\in G/G_r$ such that $H\leq g G_r g^{-1},$ and therefore $H$ fixes only the $0$-cell $g\mathbf{u}$. Then  $\mathcal{C}_G''$ is an $\uE G$ and $\dim ( \mathcal{C}''_G)^{sing}=0.$

Let us now prove the previous claim. We have to show that the space $\mathcal{C'}_G$ is simply connected and acyclic.  Since the removed cells where attached to each other, the space remains simply connected.
By construction of the push-out, we have a Mayer-Vietoris exact sequence which locally looks like
$$\dots\to \Hop_2(\ol{\mathcal{C}}_{A*B})\oplus \Hop_2(\ol{\mathcal{C}}_{G_r})\to \Hop_2(\ol{\mathcal{C}}_G)\to \Hop_1 S^{1}\to \Hop_1(\ol{\mathcal{C}}_{A*B})\oplus \Hop_1(\ol{\mathcal{C}}_{G_r})\to\dots $$
Since $\phi$ is injective at the $\pi_1$ level, the map $\Hop_1 S^{1}\to \Hop_1(\ol{\mathcal{C}}_{A*B})\oplus \Hop_1(\ol{\mathcal{C}}_{G_r})$ is inyective. Moreover, by construction of $\mathcal{C}_G',$ for $i\geq 2,$ $\Hop_i(\ol{\mathcal{C}'}_G)=\Hop_i(\ol{\mathcal{C}}_{A*B}).$ Hence for $i\geq 2,$ $\Hop_i(\ol{\mathcal{C}'}_G)=\Hop_i(\ol{\mathcal{C}}_{A*B})=\Hop_i(\ol{\mathcal{C}}_G)=0.$ Since the boundary of  $\mathbf{\alpha}$ and $\sqrt[A*B]{r}\mathbf{\alpha}$ in the 1-skeleton is the same, the identification do not change the $\Hop_1(\ol{\mathcal{C}}_G).$ Then the space  $\mathcal{C}_G'$ is acyclic.

Let $(C^A, d^A)$ and $(C^B,d^B)$ denote the integral cellular chain complexes of  ${\mathcal{C}}_A$ and ${\mathcal{C}}_B,$ the universal covers of $\ol{\mathcal{C}}_A$ and $\ol{\mathcal{C}}_B,$  Then the integral cellular chain complex $(C,d)$ of ${\mathcal{C}''}_G$ has the form at the module level
\begin{itemize}
\item $C_i\cong(\Z G\otimes_{\Z A} C_i^A)\oplus (\Z G\otimes_{\Z B} C_i^B) \text{ for } i\geq 2,$

\item $C_1\cong (\Z G\otimes_{\Z A} C_1^A)
\oplus (\Z G\otimes_{\Z B} C_1^B)\oplus \Z[G\times\{e,f\}],$

\item $C_0\cong (\Z G\otimes_{\Z A} C_0^A)\oplus (\Z G\otimes_{\Z B} C_0^B)\oplus \Z G/G_r .$
\end{itemize}
and $d_k\colon C_k\to C_{k-1}$ are induced by $(1\otimes d_k^A)\oplus (1\otimes d_k^B)$ for $k\geq 2$.

\subsection{Hempel Groups and Cohen-Lyndon asphericity}

In this subsection we discuss a family of two-relator groups wich include one-relator quotients of fundamental groups of surfaces. Originally, Hempel \cite{Hempel90} and Howie \cite{Howie04} studied one-relator quotients of the fundamental groups of orientable surfaces from a topological point of view. Hempel showed that such a group is an \emph{HNN}-extension of a one relator group. These groups have recently attracted the attention of various authors, see \cite{ADL}, \cite{Bogopolski} or \cite{HowieSaeed}.

In \cite{ADL} a family of two-relator groups with a presentation of the form
\begin{equation}\label{eq:tworelpres}
\prs{x,y,z_1,\dots,z_d}{[x,y]\cdot u, r},
\end{equation}
where $d\geq 1,$ $u\in \gp{z_1,\dots,z_d}{\quad}$ and $r\in \gp{x,y,z_1,\dots,z_d}{\quad},$ was studied.

By the classification of closed surfaces, any closed surface $S$ with Euler characteristic $\leq -2$ can be decomposed as a connected sum of a torus and projective planes or a torus and more tori. In any case, the fundamental group of $S$ admits a presentation of the form $\prs{x,y,z_1,\dots,z_d}{[x,y]\cdot u}$ with $u\in \gp{z_1,\dots,z_d}{\quad}.$

We will restrict to a subfamily of these two-relator groups, which we called {\it Hempel groups}. The definition involves some technicality, but one may think a Hempel group as a group admitting a presentation as \eqref{eq:tworelpres} where $r$ cannot be replaced by a positive power $x.$ If $Y$ is a subset of a set $X$, we say that $g\in \gp{X}{\;}$ {\it involves} some element of $Y$ if there is some element of $Y$ in the freely reduced word in $X$ representing $g.$

\begin{Def}\label{D:hempel}
Let $d\geq 1,$ $F=\gp{x,y,z_1,\dots, z_d}{\quad},$ and for $f\in F,$ $i\in \Z$ we denote $\lconj{y^i}{f}$ by $\lconj{i}{f}.$ Let also $X_1\coloneqq \{\lconj{1}{x}\}\cup\{\lconj{i}{z_j}:j=1,\dots,d; i=0,1,\dots\}$, $u\in\gp{z_1,\dots,z_d}{\quad}$ and $r\in F.$ We say that $r$ is a {\it Hempel relator} for the presentation $\prs{x,y,z_1,\dots,z_d}{[x,y]\cdot u},$ and that  $\prs{x,y,z_1,\dots,z_d}{[x,y]u,r}$ is a {\it Hempel presentation,} if the following conditions hold:
\begin{enumerate}
\item[{\normalfont (H1).}] The element $r$ belongs to $\gp{X_1}{\quad}\leqslant F.$
\item[{\normalfont (H2).}] In $\gp{X_1}{\quad},$ $r$ is not conjugate to any element of $\gen{(\lconj{0}{u})^{-1}\cdot (\lconj{1}{x})}.$
\item[{\normalfont (H3).}] With respect to $X_1,$ $r$ is cyclically reduced.
\item[{\normalfont (H4).}] With respect to $X_1,$ $r$ {\it involves} some element of $\{\lconj{0}{z_1},\dots,\lconj{0}{z_d}\}.$
\end{enumerate}

A group $G$ admitting a Hempel presentation will be called a {\it Hempel group.}
\end{Def}

\begin{Lem}[{\cite[Lemma 5.4]{ADL}}]\label{L:hempelrel}
Let $d\geq 1,$ $F=\gp{x,y,z_1,\dots, z_d}{\quad},$ $r\in F,$ and $u\in \gp{z_1,\dots,z_d}{\quad}.$ Then there exists $w\in F,$ $v\in \gen{\lconj{F}{([x,y] u)
}}$ and $\alpha\in \mathop{Aut}(F)$ that fixes $[x,y]$ and $z_1,\dots z_d$ such that $r'=\lconj{w}{(v\alpha(r))}$ is either a non-negative power of $x$ or a Hempel relator for $\prs{x,y,z_1,\dots,z_d}{[x,y]u}.$ \hfill\qed
\end{Lem}

\begin{Ex}
Let $S$ be the fundamental group of an orientable, compact, connected, closed surface of genus $g\geq 2$, and $s$ an element of $S$. Then $S/\gen{\lconj{S}{s}}$ is a Hempel group.

The group $S$ has presentation  $\prs{x_1,y_1,\dots, x_g,y_g}{[x_1,y_1]\cdots[x_g,y_g]}.$ Let $F=\gp{x_1,y_1,x_g,y_g}{\quad}$ and $r\in F$ such that the image of $r$ under the natural projection to $S$ is $s.$

We apply the process described in Lemma \ref{L:hempelrel} to $r$ and the above presentation of $S$. If $r'$ is a Hempel relator, then $S/\gen{\lconj{S}{s}}$is a Hempel group; if not, $r'=x_1^m.$ In an analogous way, we can apply the same process to $x_1^m$ and the presentation $\gp{x_2,y_2,\dots, x_g,y_g,x_1,y_1}{[x_2,y_2]\cdots[x_g,y_g][x_1,y_1]}.$
Since $\alpha$ fixes $x_1$,  we obtain a Hempel relator  and hence $S/\gen{\lconj{S}{s}}$ is a Hempel group.
\end{Ex}

In \cite[Notation 6.2]{ADL} is shown that Hempel groups are \emph{HNN}-extensions of one-relator groups, and it is implicit in \cite[Theorem 7.3]{ADL}, that Hempel presentations are aspherical in the sense of \ref{ss:aspherical}.

Instead of invoking  this results, we are going to show something stronger, namely that Hempel groups are Cohen-Lyndon aspherical, which is one of the strongest ways of being aspherical.

In \cite{CohenLyndon}, D.E. Cohen and R.C. Lyndon showed that the normal subgroup generated by a single element $r$ of a free group $F$ has free basis form of certain sets of conjugates of $r.$

The following definition is not standard, but suits with our objectives. See \cite[III.10.7]{LyndonSchupp}.
\begin{Def}\label{D:cla}
Let $F$ be a free group and $R$ a subset of $F.$ Let $G\coloneqq F/\gen{\rconj{F}{R}}$ and  for each $r\in R,$ $G_r\coloneqq \mathbf{C}_F(r)/\gen{r}=\gp{\sqrt[F]{r}}{r}$ a finite cyclic subgroup of $G$. Then $R$ is {\it Cohen-Lyndon aspherical} in $F$ if and only if
\begin{enumerate}[{\normalfont (i).}]
\item $\gen{\lconj{F}{R}}$ has a basis of conjugates of $R$;
\item\label{it:relationmodule} $\gen{\lconj{F}{R}}_{\text{ab}}$ with the left $G$-action induced by the left $F$-conjugation action $\gen{\lconj{F}{R}}$ is naturally isomorphic to $\oplus_{r\in R}\Z[G/G_r].$
\end{enumerate}
\end{Def}

\begin{Def}
A {\it Cohen-Lyndon aspherical presentation} $\prs{X}{R}$ is a presentation where  $R$ is a Cohen-Lyndon aspherical subset of $\gp{X}{\quad}.$

A {\it Cohen-Lyndon aspherical group} is a group that admits a Cohen-Lyndon aspherical presentation.
\end{Def}

As mentioned above, the first example of Cohen-Lyndon groups are one-relator groups \cite{CohenLyndon}. A {\it one-relator group} is a group which admits a presentation $\prs{X}{R}$ with $R$ having a single element.

\begin{Thm}[The Cohen-Lyndon theorem] \label{CL}
If $F$ is a free group and $r\in F-\{1\},$ then $\{r\}$ is Cohen-Lyndon aspherical in $F.$\qed
\end{Thm}

In \cite{CCH}, Chiswell, Collins and Huebschmann proved this theorem using the Magnus induction. Let us briefly recall this technique (also called \emph{Magnus breakdown}), which is probably the main tool to study one-relator groups, and was used by Magnus to prove the celebrated Freiheitssatz (see \cite[Section IV.5]{LyndonSchupp} for a reference).

\begin{Rev}[Magnus Induction]
Let $\prs{X}{r}$ be a one-relator presentation of a group $G.$ A {\it Magnus subgroup} $G$ with respect to the presentation $\prs{X}{r}$ is a subgroup generated by a subset $Y$ of $X$ such that $r$ is not a word on $Y^{\pm 1};$ Magnus' Freiheitssatz states that this subgroup is freely generated by $Y.$

The Magnus breakdown process takes a one-relator group and express it as (a subgroup of) an \emph{HNN}-extension of a one-relator group with a ``simpler relation''. More precisely, given a one-relator group $G$, there exists a finite sequence of groups $G_0,G_1,\dots, G_n,$ such that
\begin{enumerate}
\item $G_n=G$ and $G_0$ is a cyclic group;
\item $G_0,G_1,\dots, G_n,$ are one-relator groups;
\item For $i=1,\dots, n,$ $G_i$ is a subgroup of an \emph{HNN}-extension of $G_{i-1}$  where the associated subgroups are Magnus subgroups with respect to the same presentation.
\end{enumerate}
\end{Rev}

In \cite{CCH}, Chiswell, Collins and Huebschmann give a proof of the Cohen-Lyndon theorem using the Magnus induction and two results which respectively allow to construct new Cohen-Lyndon aspherical presentations from a given one by \emph{HNN}-extensions:

\begin{Thm} [{\cite[Theorem 3.4]{CCH}}] \label{T:CLAcch}
Let $F$ be a free group and $R$ a subset of $F.$ Let $\{x_{1},\dots, x_{d}\}$ and $\{y_{1},\dots, y_{d}\}$ be two subsets in $F$ such that $\{x_{1},\dots,x_d\}$ freely generates $\gen{x_{1},\dots,x_d}$ and $\{y_{1},\dots, y_{d}\}$ freely generates $\gen{y_{1},\dots, y_{d}}$ and $$\gen{x_{1},\dots,x_d}\cap\gen{\lconj{F}{R}}=\gen{y_{1},\dots,y_d }\cap \gen{\lconj{F}{R}}=\emptyset.$$
Let $t$ be a symbol.

Then $R^{+}\coloneqq R\cup\{(\lconj{t}{x_i})y_i^{-1}: i\in 1,\dots,d\}$ is Cohen-Lyndon aspherical in $F*\gp{t}{\quad}$ if and only if $R$ is Cohen-Lyndon aspherical in $F.$ \hfill\qed
\end{Thm}

These authors provide also a method to simplify presentations:

\begin{Lem}[{\cite[Lemma 5.1]{CCH}}]\label{L:claelim}
Let $F$ be a free group, $R$ a subset of $F*\gp{x}{\quad}$ and $f\in F.$ Let $\phi:F*\gp{x}{\quad}\to F,$ $x\mapsto f,$ and $h\mapsto h$ for all $h\in F.$
If $R^+\coloneqq R\cup\{xf^{-1}\}$ is Cohen-Lyndon aspherical in $F*\gp{x}{\quad}$, then $\phi(R)$ is Cohen-Lyndon aspherical in $F.$\hfill\qed
\end{Lem}

We are going to show that Hempel groups are Cohen-Lyndon aspherical and hence they lie in our family $\mathcal{G}_{cct}.$

\begin{Thm}
\label{Hempel}
Let $\prs{x,y,z_1,\dots,z_d}{[x,y]u,r}$ be a Hempel presentation of a certain group.

Then $\{[x,y]u, r\}$ is Cohen-Lyndon aspherical in $F=\gp{x,y,z_1,\dots, z_d}{\quad}$ and $\gp{x,y,z_1,\dots,z_d}{[x,y]u,r}$ is an HNN-extension of a one-relator group where the associated subgroups are Magnus subgroups.
\end{Thm}
The second part of the Theorem was proved in \cite[Notation 6.2]{ADL}.

\begin{proof}
For $f\in F,$ $i\in \Z$ we denote $\lconj{y^i}{f}$ by $\lconj{i}{f}.$  Let $$\lconj{[i,j]}{Z}\coloneqq \{\lconj{i}{z_1},\dots, \lconj{i}{z_d},\dots,\lconj{j}{z_1},\dots,\lconj{j}{z_d}\}.$$ We simply write $\lconj{i}{Z}$ for $\lconj{[i,i]}{Z}.$

By (H1) there exists a least integer $\nu$ such that $r$ lie in the subgroup $\gp{\lconj{1}{x},\lconj{[0,\nu]}{Z}}{\quad}.$ Since $\lconj{1}{x}=\lconj{0}{u}\cdot x,$ we can identify
$$\gp{\lconj{1}{x},\lconj{[0,d]}{Z}}{r}$$
 with
$$\gp{x,\lconj{[0,d]}{Z}}{r},$$ and thus view $r$ as an element of a free group with two specified free-generating sets. By (H2), $r$ is not conjugate to any element of $\gen{x}$ and hence with respect to $\{x\}\cup \lconj{[0,\nu]}{Z},$  $r$ involves some element of $\lconj{\nu}{Z}.$
By (H4), with respect to $\{\lconj{1}{x}\}\cup \lconj{[0,\nu]}{Z},$ $r$ involves some element of $\lconj{0}{Z}.$

Let
\begin{align*}
G_{[0,\nu]} &\coloneqq \gp{x,\lconj{[0,\nu]}{Z}}{r}=
\gp{\lconj{1}{x},\lconj{[0,\nu]}{Z}}{r},\\
G_{[0,(\nu-1)]} &\coloneqq \gp{x ,\lconj{[0,(\nu-1)]}{Z}}{\quad},\\
G_{[1,\nu]} &\coloneqq \gp{\lconj{1}{x} ,\lconj{[1,\nu]}{Z}}{\quad}.
\end{align*}

By Magnus' Freiheitssatz, the natural maps form $G_{[0,(\nu-1)]}$ and $G_{[1,\nu]}$ to $G_{[0, \nu]}$ are injective.

There is an isomorphism $y\colon G_{[0,(\nu-1)]}\to G_{[1, \nu]}$ given by ${x}\mapsto \lconj{1}{x}$ and $\lconj{i}{z_*}\mapsto \lconj{i+1}{z_*},$ and we can form the \emph{HNN}-extension $G_{[0,\nu]}\mathop{\ast}(y\colon G_{[0,(\nu-1)]}\to G_{[1, \nu)]})$ which gives us the group (recall that $\lconj{1}{x}=\lconj{0}{u}\cdot{x}$)
\begin{equation}\label{eq:HNNpres}\gp{x,y,\lconj{[0,\nu]}{Z}} {r,\, \lconj{y}{x}=\lconj{0}{u}\cdot x,\, (\lconj{y}{(\lconj{i}{z_*})}
=\lconj{i-1}{z_*}: i= 1,\dots \nu)}.\end{equation}

\noindent By the Cohen-Lyndon theorem, $\{r\}$ is Cohen-Lyndon aspherical in $\gp{x,\lconj{[0,\nu]}{Z}}{\quad}.$ Now by Theorem  \ref{T:CLAcch}, $\{r,\lconj{y}{x}=\lconj{0}{u}\cdot x\}\cup \{\lconj{y}{(\lconj{i}{z_*})}=\lconj{i-1}{z_*}: i= 1,\dots \nu)\}$ is Cohen-Lyndon aspherical in $\gp{x,y,\lconj{[0,\nu]}{Z}}{\quad}.$

Applying repeteadly Lemma \ref{L:claelim} we can eliminate the free factor $\gp{\lconj{[1,\nu]}{Z}}{\quad}$ and the set of relations $\{\lconj{y}{(\lconj{i}{z_*})}=\lconj{i-1}{z_*}: i= 1,\dots \nu\}$, so we obtain that $\{r,[x,y]u\}$ is Cohen-Lyndon aspherical in $\gp{x,y,z_1,\dots,z_d}{\quad}.$
\end{proof}

\subsection{Other examples}

In \cite{LuckStamm}, L\"{u}ck and Stamm are interested in groups where the finite subgroups satisfy a property similar to our condition (C). They provide two families of examples, which we adapt to satisfy condition (C).

\begin{itemize}
\item Extensions $1\to \Z^n\to G\to C\to 1,$ where $C$ is a finite and cyclic and the conjugation action of $C$ in $\Z^n$ is free outside   $0\in \Z^n$. See \cite[Lemma 6.3]{LuckStamm}.


\item Cocompact NEC-groups that do not contain finite dihedral subgroups. By \cite[Lemma 4.5]{LuckStamm}, they satisfy our property (C). An example of such groups are groups with presentation \begin{equation}\prs{a_1,\dots,a_r,c_1,\dots,c_t}{c_1^{\gamma_1}=\dots=c_t^{\gamma_t}=c_1^{-1}\cdots c_{t}^{-1}a_1^2\cdots a_r^2 = 1 }\end{equation} where $\gamma_i\geq 2$ for $i=1,\dots,t.$
\end{itemize}

\section{Bredon homology}

Let $G$ be a group in $\mathcal{G}_{cct}$. By (C) there is a model for the
classifying space for proper $G$-actions $\uE G$ with 0-dimensional singular part, which allows to compute the Bredon homology of $G$. In turn, this will open the way to describe the topological part of the Baum-Connes conjecture for these groups (see Section \ref{BC}).

\subsection{Background on Bredon homology} Our main source for this section has been \cite[Section 3]{MislinValette}, while the original and main reference goes back to \cite{Bredon}.

Let $G$ be a discrete group and $\mathfrak{F}$ a non-empty family of subgroups of $G$ closed under subgroups and conjugation. The reader should have in mind the family of finite subgroups of $G,$ which we will denote by $\mathfrak{Fin} (G)$, or simply $\mathfrak{Fin}$ if the group is understood.

The orbit category $\mathfrak{D}_{\mathfrak F}(G)$ is the category whose objects are left coset spaces $G/K$ with $K\in \mathfrak F,$ and morphism sets $\mor(G/K,G/L)$ given by the $G$-maps $G/K\to G/L;$ this set can be naturally identified with \begin{equation}\label{eq:mor}(G/L)^K\coloneqq\{gL\in G/L:KgL=gL\}=\{gL\in G/L:g^{-1}Kg\leqslant L\},\end{equation}
that is, the cosets fixed by the $K$-multiplication action on $G/L.$

Let $G$-$\Mod_{\mathfrak F}$ and $\Mod_{\mathfrak F}$-$G$ be respectively the category of covariant and contravariant functors  $\mathfrak D_{\mathfrak F}(G)\to \mathfrak{Ab}$ from the orbit category to the category $\mathfrak{Ab}$ of abelian groups. Morphisms in $G$-$\Mod_{\mathfrak F}$ and $\Mod_{\mathfrak F}$-$G$ are natural transformations of functors. Notice that if $\mathfrak F$ consists only on the trivial group then $G$-$\Mod_{\mathfrak F}$ and $\Mod_{\mathfrak F}$-$G$ are the usual categories of right and left $\Z G$-modules.

The category $\Mod_{\mathfrak F}$-$G$ is abelian \cite[page 8]{MislinValette}, and an object $P\in \Mod_{\mathfrak F}$-$G$ is called \emph{projective} if the functor
$$\mor(P,-)\colon \Mod_{\mathfrak F}\text{-}G\to \mathfrak{Ab}$$ is exact. Every $M\in \Mod_{\mathfrak F}$-$G$ admits a projective resolution and  projective resolutions are unique up to chain homotopy.

Let $M\in \Mod_{\mathfrak F}$-$G$ and $N\in G$-$\Mod_{\mathfrak F}$. By definition $M\otimes_{\mathfrak F} N$ is the abelian group $$\left(\sum_{G/K\in \mathfrak{D}_{\mathfrak F}(G)} M(G/K)\otimes_{\Z} N(G/K)\right)/\sim,$$
where $\sim$ is the equivalence relation generated by $$M(\phi)(m)\otimes n \sim m\otimes N(\phi) (n),$$ with $\phi\in \mor (G/K,G/L)$ and $m\in M(G/L), n\in N(G/K).$

Let now $\underline{\Z}$ denote the constant functor which assigns to each object the abelian group $\Z$. Then, for example
$$\underline{\Z}\otimes_{\mathfrak F}N=\sum_{G/K\in \mathfrak{D}_{\mathfrak F}(G)} N(G/K)/\sim$$
In this context, $n\sim m$ if and only if $n\in N(G/L)$ and $m= N(\phi)(n)$ for $\phi\in \mor (G/K,G/L).$ Hence $$\underline{\Z}\otimes_{\mathfrak F}N=\text{colim}_{G/K\in \mathfrak{D}_{\mathfrak F}(G)} N(G/K).$$

Now $\Tor_i(-,N)$ is defined as the $i$-th left derived functor of the {\it categorical tensor product functor }$-\otimes_{\mathfrak F}N\colon\Mod_{\mathfrak F}\text{-}G\to \mathfrak{Ab}$, and the {\it Bredon homology groups of $G$ with coefficients in $N$ }\index{Bredon homology} $\in G$-$\Mod_{\mathfrak F}$ are given by $$\Hop^{\mathfrak F}_i(G;N)\coloneqq \Tor_i(\underline{\Z},N), \; i\geq 0, $$
where $\underline{\Z}$ denotes the constant functor which assigns to each object the abelian group $\Z$.

For example, one has \begin{equation}\label{eq:H0bredon}
\Hop^{\mathfrak F}_0(G;N)=\underline{\Z}\otimes_{\mathfrak F}N=\text{colim}_{G/K\in \mathfrak{D}_{\mathfrak F}(G)} N(G/K).\end{equation}

Let $X$ be a $G$-$CW$-complex such that all the $G$-stabilizers of $X$ lie in $\mathfrak F$, and $N\in G$-$\Mod_{\mathfrak F}$;  then one defines {\it Bredon homology groups of $X$ with coefficients in $N$ } as $$\Hop^{\mathfrak F}_i(X;N)\coloneqq \Hop_i(\underline{C_*(X)}\otimes_{\mathfrak F}N), \; i\geq 0,$$
where $\underline{C_j(X)}\colon \mathfrak D_{\mathfrak F}(G)\to \mathfrak{Ab}$ is defined by $G/H\mapsto \Z[\Delta_j^H]$, being the latter the free abelian group on the $j$-dimensional cells of $X$ fixed by $H.$ It can be shown  (\cite[Section 3]{MislinValette}) that $\underline{C_*(E\mathfrak{ F})}\to \underline{\Z}$ is a projective resolution and hence the functors $\Hop^{\mathfrak F}_i(E\mathfrak F;-)$ and $\Hop^{\mathfrak F}_i(G;-)$ are equivalent.

\subsection{Bredon homology of groups in $\mathcal{G}_{cct}$}\label{Bredon}

From now on, we concentrate on the case $\mathfrak{F}=\mathfrak{Fin}(G)$ the class of finite subgroups of $G$. Let $R_{\C}$ denote the covariant functor $\mathfrak{D}_{\mathfrak F}(G)\to \mathfrak{Ab}$ which sends every left coset space $G/H$ to the underlying abelian group of the complex representation ring $R_{\C}(H)$. Recall that every $G$-map $f\colon G/H \to G/K,$ with $f(H)=gK$, gives rise to a group homomorphism $f'\colon H\to K,$ $h \mapsto g^{-1}hg$,  which is unique up to conjugation in $K.$ Since inner automorphism act trivially on the complex representation ring, the functor $R_{\C}$ sends the map $f$ to the homomorphism $R_{\C}(H)\rightarrow R_{\C}(K)$ induced by $f'$.

The main goal of this section is the description of the Bredon homology with coefficients in the complex representation ring $R_{\C}$ for a group $G$ in $\mathcal{G}_{cct}$, and the reason of the choice of the category of coefficients comes from its role in the context of Baum-Connes conjecture.

To undertake our problem, we first characterise the special shape of the orbit category $\mathfrak{D}_{\mathfrak F}(G)$ of these groups. So, in all the sequel $G$ will be a group and $\{G_\lambda\}_{\lambda\in \Lambda}$ a family of non-trivial cyclic subgroups of $G$ that satisfy (C).
For $\lambda \in \Lambda,$ let $\mathfrak{F}(\lambda)\coloneqq \{K : K\leqslant G_\lambda \},$ the set of subgroups of $G_\lambda.$

By (C), the set $\mathfrak{F}$ of finite subgroups of $G$ is the union up to conjugation of the sets $\mathfrak{F}(\lambda),$ that is $$\mathfrak{F}=\{ \rconj{g}{K_\lambda}: \lambda \in \Lambda, \{1\}\neq K_\lambda \leqslant G_\lambda, g\in G/G_\lambda \}\cup\{\{1\}\}.$$

For $\lambda,\mu\in \Lambda,$ $a\in G/G_\lambda$, $b\in G/G_\mu$ and $\{1\}\neq K_\lambda\leqslant G_\lambda,$ $\{1\}\neq L_\mu\leqslant G_\mu$, suppose that $\mor(G/ \rconj{a}{K_\lambda},G/\rconj{b}{L_\mu} )\neq \emptyset ;$
then by \eqref{eq:mor} there is some $g\in G$ such that $g\rconj{a}{K_\lambda}g^{-1}\leqslant \rconj{b}{L_\mu}$
 and the uniqueness in condition (C) implies that $\lambda=\mu.$ Moreover, since $G_\lambda$ is cyclic, we have $K\leqslant L$ by the structure of the subgroups of a cyclic group.

It is easy to show the converse, that is

\begin{equation}\label{eq:mor2}\mor(G/ \rconj{a}{K_\lambda},G/\rconj{b}{L_\mu} )\neq \emptyset \Leftrightarrow \lambda=\mu \text{ and } K_\lambda \leqslant L_\mu
\end{equation}
and for the case of the trivial group  we have
\begin{equation}\label{eq:mor3} \mor(G/\{1\},G/\rconj{b}{L_\mu})=G/\rconj{b}{L_\mu}\text{ and }\mor(G/\rconj{a}{K_\lambda},G/\{1\})=\emptyset
\end{equation}
Hence, the only non-trivial morphisms in this orbit category are inclusions.

Now we are ready to compute the 0-th Bredon homology groups we are interested in.  By \eqref{eq:H0bredon}   $$\Hop_0^{\mathfrak{F}}(G;R_\C)= \text{colim}_{G/K\in \mathfrak{D}_{\mathfrak{F}(G)}} R_\C(K).$$
Hence, if $\Lambda$ is non-empty we have
$$\Hop_0^{\mathfrak{F}}(G;R_\C)=\prod_{\lambda\in \Lambda}(\text{colim}_{K\in \mathfrak{F}(\lambda) }R_\C(K)).$$

By \eqref{eq:mor2} and \eqref{eq:mor3} all the subgroups in $\mathfrak{F}(\lambda)$ are, up to conjugation, cyclic subgroups of the cyclic $G_\lambda,$ and the morphisms are given by restriction of representations, we have
$$\Hop_0^{\mathfrak{F}}(G;R_\C) =\prod_{\lambda\in \Lambda}R_\C(G_\lambda).$$

In the case $\Lambda$ is empty, we have that $\Hop_0^{\mathfrak{F}}(G;R_\C)=R_\C(\{1\})=\Z.$

Now we deal with the higher homology groups. As seen before, a classifying space for proper actions $\uE G$ can be used to compute $\Hop_*^{\mathfrak{F}}(G;R_\C).$ For example, it is known that the virtually free groups are exactly the groups which admit a tree as a model of $\uE G$, and their Bredon homology with coefficients in the complex representation ring has been described by Mislin in \cite[Theorem 3.17]{MislinValette}.

We denote  by $\textrm{\underline{B}}G$ the orbit space $(\uE G)/G$, sometimes called ``classifying space for $G$-proper bundles" (see \cite[Appendix 3]{BaumConnes} for the motivation of the name). Note that an $n$-dimensional model for $\uE G$ produces $n$-dimensional models for  its orbit space. Moreover, we define $\dim(\uE G)^{\text{sing}}$ to be the minimum of $\dim(X^{\text{sing}})$ where $X$ is a model for $\uE G$.

The following result relates the Bredon homology of $\uE G$ and the ordinary homology of $\textrm{\underline{B}}G$:

\begin{Lem}[{\cite[Lemma  3.21]{MislinValette}}]\label{L:mislim1}
Let $G$ be an arbitrary group. Then there is a natural map
$$\Hop_i^{\mathfrak {Fin}}(\uE G; R_\C)\to \Hop_i(\emph{\underline{B}}G;\Z) $$
which is an isomorphism in dimension $i>\dim (\uE G)^{\text{sing}}+1$ and injective in dimension $i=\dim (\uE G)^{\text{sing}}+1.$\qed
\end{Lem}

Now we are ready to state the main result of this section.

\begin{Thm}\label{T:clabredonhom}
Let $G$ be a group on the class $\mathcal G_{cct}$ and let $\{G_\lambda\}_{\lambda\in \Lambda}$ be the subgroups for which condition (C) holds.
Then
\begin{enumerate}[\normalfont(i).]
\item $\Hop_i^{\mathfrak{F}}(G;R_\C)=\Hop_i(\emph{\underline{B}}G; \Z)$ for $i\geq 2.$

\item $\Hop_1^{\mathfrak{F}}(G;R_\C)=(G/\Tor(G))_{\text{ab}} $ where $\Tor(G)$ denotes the subgroup of $G$ generated by the torsion elements.

\item $\Hop_0^{\mathfrak{F}}(G;R_\C)=\prod_{\lambda \in \Lambda }R_\C (G_\lambda),$ if $\Lambda\neq \emptyset$ or  $\Hop_0^{\mathfrak{Fin}}(G;R_\C)=\Z,$ if $\Lambda=\emptyset.$

\end{enumerate}
\end{Thm}
\begin{proof}
(i) follows from Lemma \ref{L:mislim1}, since by Proposition \ref{bar-construction}  $\dim (\uE G)^{\text{sing}}=0.$
(iii) has been proved at the beginning of this section, so it only remains to prove (ii).

By Proposition \ref{bar-construction}, we have a model $X$ of $\uE G$ such that $\dim(\uE G)^{\text{sing}}=0.$  Moreover, the vertex set of $X$ is in bijection with  $G\sqcup\{g G_\lambda :\lambda \in \Lambda, g\in G\}$. Hence, we have the following exact augmented cellular chain complex of $\Z G$-modules
$$\dots \to \oplus_{J_2} \Z G \to \oplus_{J_1} \Z G\to  \Z G \oplus (\oplus_{\lambda\in \Lambda}\Z [G/G_\lambda]) \to  \Z \to  0$$
where $J_i$ is the number of free $G$-orbits of $i$-dimensional cells in $X.$

Also using the cellular structure of $X,$ we have the following projective resolution of $\underline{\Z}$ in $\Mod_{\mathfrak{Fin}}$-$G$
$$\dots \to \underline{C_2(X)} \to \underline{C_1(X)} \to  \underline{C_0(X)} \to  \underline{\Z} \to  0$$
where
$$\underline{C_i(X)}\colon \mathfrak{D}_{\mathfrak{Fin}}(G)\to \mathfrak{Ab}, \quad \underline{C_i(X)}(G/H)=\begin{cases} 0 &\text{ if }H\neq\{1\}\\\Z G^{J_i}& \text{ if } H=1\end{cases}$$
for $i=1,2,\dots.$ For $i=0$ we have, $\underline{C_0(X)}\colon \mathfrak{D}_{\mathfrak{Fin}}(G)\to \mathfrak{Ab},$ $\underline{C_0(X)}(G/H)=\Z [G/G_\lambda]$ if $H\neq 1$ where $G_\lambda$ is the unique subgroup that contains a conjugate of $H,$  and $$\underline{C_0(X)}(G/\{1\})=\Z G \oplus (\oplus_{\lambda\in \Lambda}\Z [G/G_\lambda]).$$

Let us write $\underline{C_i}$ for $\underline{C_i(X)}\otimes_{\mathfrak{F}}R_\C$ and $C_*$ for the ordinary cellular chain complex $C_*(X/G).$ Then we have:

$$\underline{C_0}\cong\Z \oplus\left( \oplus_{\lambda\in \Lambda} R_C(G_\lambda)\right),\qquad C_0\cong\Z \oplus\left( \oplus_{\lambda\in \Lambda} \Z \right).$$

We have the diagram in Figure~\ref{fig:diag},

\begin{figure}[ht]
\centerline{
\xymatrix{
 & & \ker(\pi)\ar[d]^{\text{mono}} \ar[rr]^{\hspace{-0.7cm}\cong}&& \ker(\epsilon)=\prod_{\lambda\in \Lambda}\tilde{R}_{\C}(G_\lambda)\ar[d]^{\text{mono}}\\
\underline{C_2}\ar[r]^{\underline{d_2}}\ar[d]^{=} & \underline{C_1}\ar[r]^{\underline{d_1}}\ar[d]^{=} & \underline{C_0}\ar[rr]^{\hspace{-0.7cm}\text{epi}}\ar[d]^{\pi} && \prod_{\lambda\in \Lambda} R_\C(G_\lambda)\ar[d]^{\epsilon}\\
 C_2\ar[r]^{d_2} & C_1\ar[r]^{d_1} & C_0\ar[rr]^{\hspace{-0.7cm}\text{epi}} & & \prod_{\lambda\in \Lambda} \Z}
}\label{fig:diag}\caption{Comparing the two chain complexes}
\end{figure}
\noindent where $\pi=(\mathop{Id},\epsilon)$ and $\epsilon\colon \prod R_\C(G_\lambda)\to \prod \Z$ is the augmentation. Now we are ready to check the statements of the theorem.

By \cite[Proposition 3]{LearyNucinkis} (see also \cite{Armstrong}), $\pi_1(\underline{\textrm{B}}G)= G/\Tor(G)$ . Hence, $\Hop_1(\underline{\textrm{B}}G;\Z)\cong (G/\Tor (G))_{\text{ab}}$. By Lemma \ref{L:mislim1}, the map $\Hop_1^{\mathfrak {F}}(\uE G; R_\C)\to \Hop_1(\underline{\textrm{B}}G;\Z) $ is injective. We have to show that it is surjective too.

The diagram shows that $\ker \underline{d_1}=\ker d_1,$ thus $\ker\underline{d_1}/ \mathop{\text{Im}} \underline{d_2}\to \ker{d_1}/ \mathop{\text{Im}} {d_2}$ is onto.
\end{proof}

\begin{Rem}

The model for $\uE G$ of \cite[4.11]{Luck04} is given by a $G$-pushout which involve $EG$ and $EG_{\lambda}$ for $\lambda \in \Lambda$, so a Mayer-Vietoris argument may give some extra information about the higher homology groups $\Hop_i(\textrm{\underline{B}}G; \Z)$ appearing in Theorem \ref{T:clabredonhom} (i).
\end{Rem}

For the concrete families described above, we may give sharper statements:

\begin{Cor}\label{C:libredonhom}
Let $A$ and $B$ be locally indicable groups and $r\in A*B$ not conjugate to an element of $A$  nor of $B$.

Let $G\coloneqq (A*B)/\gen{\rconj{A*B}{r}}$ and  let $G_r$ be the cyclic subgroup of $G$ generated by the image of $\sqrt[A*B]{r}.$

Then
\begin{enumerate}[\normalfont(i).]
\item $\Hop_i^{\mathfrak{F}}(G;R_\C)=\Hop_i(A;\Z)\oplus \Hop_i(B;\Z)$ for $i>2.$

\item $\Hop_2^{\mathfrak{F}}(G;R_\C)=\Hop_2(\emph{\underline{B}}G;\Z)$.

\item $\Hop_1^{\mathfrak{F}}(G;R_\C)=(G/\Tor(G))_{\text{ab}} $ where $\Tor(G)$ denotes the subgroup of $G$ generated by the torsion elements.

\item $\Hop_0^{\mathfrak{F}}(G;R_\C)=R_\C (G_r)$ if $G_r \neq \{1\}$ or  $\Hop_0^{\mathfrak{Fin}}(G;R_\C)=\Z$ if $G_r=\{1\}.$
\end{enumerate}
\end{Cor}
\begin{proof}
By the discussion in the subsection \ref{ss:oreli}, $G$ and $G_r$ are under the hypotheses of Theorem \ref{T:clabredonhom}, and (ii),(iii) and (iv) follow directly. Statement (i) follows from the cellular chain complex of $\mathcal{C}''_G$ described in \ref{ss:oreli}.
\end{proof}

\begin{Cor}\label{C:clabredonhom}
Let $F$ be a free group and $R$ be subset of $F$.
Let $G\coloneqq F/\gen{\rconj{F}{R}}$ and  for each $r\in R,$ let $G_r$ be the cyclic subgroup of $G$ generated by the image of $\sqrt[F]{r}.$ Suppose that $\gen{\lconj{F}{R}}_\text{ab}\cong \oplus_{r\in R} \Z[G/G_r]$, and let $T(R)=\{r\in R:G_r\neq 1\}$. Then
\begin{enumerate}[\normalfont(i).]
\item $\Hop_i^{\mathfrak{F}}(G;R_\C)=0$ for $i>2.$

\item $\Hop_2^{\mathfrak{F}}(G;R_\C)=\Hop_2(G;\Z).$

\item $\Hop_1^{\mathfrak{F}}(G;R_\C)=(G/\Tor(G))_{\text{ab}} $ where $\Tor(G)$ denotes the subgroup of $G$ generated by the torsion elements.

\item $\Hop_0^{\mathfrak{F}}(G;R_\C)=\prod_{r\in T(R)}R_\C (G_r)$ if $T(R)\neq \emptyset$ or  $\Hop_0^{\mathfrak{Fin}}(G;R_\C)=\Z$ if $T(R)=\emptyset.$

\end{enumerate}
\end{Cor}
\begin{proof}
By the discussion in the subsection \ref{ss:aspherical}, $G$ and $\{G_r\}_{r\in T(R)}$ are again under the hypothesis of Theorem \ref{T:clabredonhom} and we have a 2-dimensional model for $\uE G$ so (i),(iii) and (iv) follow.
It only remains to prove the statement (ii). From the theorem we know that $\Hop_2^{\mathfrak{F}}(G;R_\C)=\Hop_2(\underline{B} G;\Z),$ which, by  the cellular chain complex \eqref{eq:seqaspherical} of $\mathcal{C}''$ described in \ref{ss:aspherical}, is isomorphic to the kernel of the induced map $(\Z G^{|R|}\to \Z G^{|R|+|X|} )\otimes_{\Z [G]} \Z.$

The complex $\mathcal{C}'$ described in \ref{ss:aspherical} has cellular chain complex \eqref{eq:seq} and is homotopic to $\mathcal{C}'',$ hence $\Hop_2(\underline{B} G;\Z)$ is isomorphic to the kernel of the induced map  $(\oplus_{r\in R}\Z [G/G_{r}]\to \Z G^{|X|} )\otimes_{\Z [G]} \Z,$ which is isomorphic to $\Hop_2(G,\Z).$\end{proof}

In the case of Hempel groups, we can be even more precise:

\begin{Thm} \label{Hempel:bredon}
Let $G$ be a Hempel group with presentation $\prs{x_{1},\dots,x_k}{w,r}$, with $k\geq 3,$  $w\in[x_1,x_2]\gen{x_{3},\dots,x_k}\subseteq \gp{x_{1},\dots,x_k }{\quad}=F$ and such that $r$ is a Hempel relator for $\prs{x_{1},\dots,x_{k}}{w}.$ Let  $G_r\coloneqq \mathbf{C}_F(r)/\gen{r}=\gp{\sqrt[F]{r}}{r}.$
Then
\begin{enumerate}[\normalfont(i).]
\item $\Hop_i^{\mathfrak{Fin}}(G;R_\C)=0$ for $i>2.$
\item $\Hop_2^{\mathfrak{Fin}}(G;R_\C)=\Hop_2(G;\Z)=(\gen{\rconj{F}{r}\cup\rconj{F}{w}}\cap [F,F])/[F,\gen{\rconj{F}{r}\cup\rconj{F}{w}}]$
\item $\Hop_1^{\mathfrak{Fin}}(G;R_\C)=(\gp{x_{1},\dots,x_k}{w,\sqrt[F]{r}})_{\text{ab}}.$
\item $\Hop_0^{\mathfrak{Fin}}(G;R_\C)=R_\C (G_r)$ if $G_r\neq \{1\}$ and $\Z$ in the other case.
\end{enumerate}
\end{Thm}
\begin{proof}
By Theorem \ref{Hempel} we are in the hypothesis of Corollary \ref{T:clabredonhom}.
The last equality of (ii) is the classical Hopf identity.
\end{proof}

\section{Relation with Baum-Connes conjecture}\label{BC}

Let $H$ be an aspherical group. Beyond their intrinsic interest, the results obtained in the previous section show their relevance in the context of Baum-Connes conjecture. More concretely, Corollary \ref{C:clabredonhom} will identify the equivariant version $K_i^H(\uE H)$ of the $K$-homology of $\uE H$, as defined by Davis-L\"uck in \cite{DavisLuck}. Given an arbitrary countable discrete  group $G$, the $K$-groups $K_i^G(\uE G)$, which are defined via the non-connective topological $K$-theory spectrum, can be in turn identified with the Kasparov $KK$-groups $KK_i^G(\uE G)$, which are constructed as homotopy classes of $G$-equivariant elliptic operators over $\uE G$. These homotopical invariants are related with the topological algebraic $K$-groups $K_i^{top}(C^*_r(G))$ of the reduced $C^*$-algebra of $G$, an object which is defined as the closure of a subalgebra of the Banach algebra $\mathcal{B}({l_2(G)})$ of bounded operators over the space of square-summable complex functions over the group $G$, and whose nature is thus essentially analytic. Bott periodicity holds for both the homotopical and the analytical groups, and the relationship is given, for $i=0,1$, by an index map:

$$\Theta: KK_i^G(\uE G)\to K_i^{top}(C^*_r(G)), $$

The Baum-Connes conjecture (or BCC for short) asserts that this index map is an isomorphism for every second countable locally compact group $G$. Originally stated in its definitive shape by Baum-Connes-Higson in \cite{BaumConnes}, its importance come mainly from two sources: first, it relates two objects of very different nature, being the analytical one particularly inaccessible; and moreover, it implies a number of famous conjectures, as for example Novikov conjecture on the higher signatures, or the weak version of Hyman Bass' conjecture about the Hattori-Stallings trace; see \cite[Section 7]{MislinValette} for a good survey on this topic. The conjecture BCC has been verified for an important number of groups, and in particular for the groups in the class $\mathbf{LH}\mathcal{TH}$, which is defined by means of an analytical property (see \cite[Section 5]{MislinValette} for a detailed exposition). The class $\mathbf{LH}\mathcal{TH}$ contains for example the soluble groups, finite groups and free groups; and as it also contains one-relator groups and it is closed under passing to subgroups and \emph{HNN}-extensions, Theorem \ref{Hempel} implies that Hempel groups are in $\mathbf{LH}\mathcal{TH}$.

On the other hand, every one-relator group is also aspherical (see Theorem \ref{CL} above). However, it is unknown if BCC holds for the class of aspherical groups, so it is interesting to investigate the value of the $K$-groups in both sides of the conjecture for  aspherical groups. The main goal of this paragraph is to show how the results in the previous section allow to compute the topological side of the conjecture. The key result here, owed to Mislin \cite[Theorem 5.27]{MislinValette}, is in fact a collapsed version of an appropriate Atiyah-Hirzebruch spectral sequence:

\begin{Thm}
Let $G$ be an arbitrary group such that dim $\uE G\leq 2$. Then there is a natural short exact sequence:

$$0\rightarrow R_\C(G) \rightarrow K_0^G(\uE G)\rightarrow \Hop_2^{\mathfrak{Fin}}(G;R_\C),$$ and a natural isomorphism $\Hop_1^{\mathfrak{Fin}}(G;R_\C)\simeq K_1^G(\uE G)$.

\end{Thm}

Now the description of the Kasparov aspherical groups comes straight from Corollary \ref{C:clabredonhom}, and generalises Corollary 5.28 from \cite{MislinValette}:

\begin{Prop}
Let $G$ be an aspherical group. Then $K_0^G(\uE G)$ fits in a short exact sequence $$R_\C(G)\rightarrow K_0^G(\uE G)\rightarrow \Hop_2(G;\Z)$$ that splits, and moreover there is a natural isomorphism $(G/\Tor(G))_{\text{ab}}\simeq K_1^G(\uE G)$.

\label{Kasparov}
\end{Prop}

In particular, BCC holds for Hempel groups, so we have also computed the analytical of the conjecture for these groups:

\begin{Prop}
Let $G$ be a Hempel group with presentation $\prs{x_{1},\dots,x_k}{w,r}$, with $k\geq 3,$  $w\in[x_1,x_2]\gen{x_{3},\dots,x_k}\subseteq \gp{x_{1},\dots,x_k }{\quad}=F$ and such that $r$ is a Hempel relator for $\prs{x_{1},\dots,x_{k}}{w}.$ Let  $G_r\coloneqq \mathbf{C}_F(r)/\gen{r}=\gp{\sqrt[F]{r}}{r}.$
Then $K_i^G(\uE G)\simeq K_i^{top}(C^*_r(G))$ for $i=0,1$, there is a split short exact sequence $$R_\C(G)\rightarrow K_0^G(\uE G)\rightarrow (\gen{\rconj{F}{r}\cup\rconj{F}{w}}\cap [F,F])/[F,\gen{\rconj{F}{r}\cup\rconj{F}{w}}]$$ and a natural isomorphism $(\gp{x_{1},\dots,x_k}{w,\sqrt[F]{r}})_{\text{ab}}\simeq K_1^G(\uE G)$.

\end{Prop}

\begin{proof}

It is a consequence of Theorem \ref{Hempel:bredon}.

\end{proof}

\bibliographystyle{amsplain}

\textsc{Yago Antol\'{i}n, School of Mathematics,
University of  Southampton, University Road,
Southampton SO17 1BJ, UK}

\emph{E-mail address}{:\;\;}\url{y.antolin-pichel@soton.ac.uk}

\medskip

\textsc{Ram\'{o}n Flores,
Departamento de Estad\'{i}stica, Universidad Carlos III,
Avda. de la Universidad Carlos III, 22
28270 Colmenarejo (Madrid), Spain
}

\emph{E-mail address}{:\;\;}\url{rflores@est-econ.uc3m.es}
\end{document}